\documentclass[11pt]{amsart}
\usepackage{color}
    \newcommand{\argmin}{\mathop{\mathrm{argmin}}}
    \newcommand{\AC}{\mathop{\mathrm{AC}}}
    \newcommand{\lip}{\mathop{\mathrm{lip}}}
    \newcommand{\bound}{\mathop{\mathrm{bound}}}
    \newcommand{\R}{\mathbb{R}}

    \newtheorem{theorem}{\bf Theorem}[section]
    
    \newtheorem{lemma}{\bf Lemma}[section]

    \newtheorem{corollary}{\bf Corollary}[section]
    \newtheorem{remark}{\bf Remark}[section]

\title[Hopf--Lax formula with discount]
{Hopf--Lax formula for variational problems with non--constant
discount}

\author{Juan Pablo Rinc\'{o}n--Zapatero}

\begin{document}


\begin{abstract}
We provide a Hop--Lax formula for variational problems with
non--constant discount and deduce a dynamic programming equation.
We also study some regularity properties of the value function.
\end{abstract}
\date{\today}


\maketitle


\section{Introduction}
We establish a Hopf--Lax formula for the Cauchy problem
\begin{equation}\label{Diss}
\left\{%
\begin{array}{rl}
    -v_t(x,t)+f(-v_x) + \rho(t)v(x,t)=0, & \hbox{in $\R^n\times (0,T)$;} \\
    v=g, & \hbox{on $\R^n\times \{t=T\}$,} \\
\end{array}%
\right.
\end{equation}
involving a Hamilton--Jacobi equation with a linear dissipation
term, $\rho(t) v(x,t)$, and a terminal condition at time $t=T$.
The function $f:\R^n\longrightarrow\R$ is assumed to be convex and
of class $C^2$, $\rho:[0,T]\longrightarrow (0,1]$ is continuous,
and $g:\R^n\longrightarrow\R$ is globally Lipschitz. The formula
is
\[
v(x,t)=\min_{p\in \R^n} \left\{ \int_t^T  d_t(s) \ell \big(\iota(
d_t^{-1}(s)p)\big)\, ds +  d_t(T) g\left(x+\int_t^T \iota(
d_t^{-1}(s)p)\, ds\right)\right\},
\]
where $\ell$ is the convex conjugate of $f$, $\iota=(\nabla
\ell)^{-1}$, and $  d_t(s)=\exp{(-\int_t^s \rho (r)\, dr)}$.
The formula represents a Lipschitz solution that satisfies the
Cauchy problem almost everywhere.

The classical Hopf--Lax formula applies to the case $\rho\equiv 0$
, and was given by Lax in \cite{Lax}, for $n=1$. It was extended
later to general $n$ by Hopf in \cite{Hopf}. Further
generalizations have maintained $\rho=0$ but have considered
functions $f(t,-v_x)$ also depending on time, \cite{VanTsujiSon};
or functions $f(v,-v_x)$ depending on both $v$ and $v_x$,
\cite{BarronJensenLiu}, with some additional requirements. The
case we analyze in this paper is not covered in any of these
previous works.

Actually, we find a Hopf--Lax formula that applies to more general
Hamilton--Jacobi equations associated to the calculus of
variations problems with variable discount.

These problems arise quite naturally in models of economics.
Consider for instance the following problem: an agent optimally
chooses a consumption path of a given good, with the aim of
maximizing his/her satisfaction. This is measured by an utility
function of consumption, $\ell(u)$, along a given time interval,
$[0,T]$. It is customary in the literature to postulate concavity
in the preferences of the agent, and to suppose that he/she is
impatient, in the sense that the value of the utility attained
today is higher than the utility attained tomorrow. This is the
meaning of introducing a discount factor or impatience rate in the
preferences of the agent.

Empirical studies suggest that people are more impatient about
choices in the short run than in the long run, implying that the
discount rate applied to current choices is higher than the one
applied to far--in--the future choices. Thus, the discount factor
should be taken to be non--constant. Several papers have
considered the non--constant discount case: see e.g. \cite{Barro},
\cite{Karp} or \cite{MarinNavas}\footnote{We consider only papers
on continuous time.}. In \cite{Barro}, the optimal growth model
with time--varying discount is considered, for a particular class
of utility functions. A general problem, with infinite horizon, is
analyzed in \cite{Karp}, whereas \cite{MarinNavas} considers the
finite horizon case with fixed or variable terminal time. The last
two papers use discretization and passage to the limit to find a
Hamilton--Jacobi equation that involves not only the unknown value
function, but also a non--local term involving integration along
the unknown optimal solution. We provide conditions so that the
Hamilton--Jacobi equation involves only the derivatives of the
value function and find the dynamic programming equation by direct
methods.

Given the significance of the non--constant discount preference
rate in economics, it is of interest to analyze in more detail
this
type of variational problems. First, deriving 
a Hopf--Lax formula for the solution of the variational problem
(Section  \ref{s:Hopf}); second, establishing a modified dynamic
programming equation, more amenable than the one found in previous
papers (Section \ref{s:DPE}); and third, studying the regularity
of the value function (Section \ref{s:Reg}).

\section{Variational problem with discount}
We follow the presentation in \cite{Evans}. Let the value function
\begin{equation}\label{Problem}
v(x,t)=\inf_{y\in \AC _{x,t}} \left\{ \int_t^T  d_t(s) \ell (\dot
y(s))\, ds +  d_t(T) g(y(T))\right\},
\end{equation}
where
\[
{\AC}_{x,t}=\{y:[t,T]\longrightarrow \R^n\,:\, y=y(s) \mbox{
absolutely continuous, }y(t)= x\}.
\]
A typical element of this set will be called an arc. We will
impose the following conditions.

\begin{description}
    \item[A1] $\ell:\R^n\longrightarrow \R$ is $C^2$, strictly convex, and $\displaystyle
    \lim_{|u|\to \infty} \frac {\ell (u)}{|u|}=\infty$;
    \item[A2] $g$ is globally Lipschitz in $\R^n$;
    \item[A3] $ d:[0,T]\times [0,T]\longrightarrow
    (a,1]$, with $a>0$, is Lipschitz continuous with
    $ d_t(t)=1$ for each $t$.
\end{description}

A straightforward interpretation of \eqref{Problem} has been done
in the Introduction: a single agent, with time--varying preference
rate, chooses optimally along time. In another reading, there is a
continuum of agents, each one labelled by $t\in [0,T]$; each agent
(or generation $t$) applies a possibly different discount factor,
$ d_t$, in the calculation of the utility flow from $t$ onwards.
At time $T$ the optimization process finishes and agent $T$
derives utility $g(y(T))$ (or ``scrap value"). The aim of each
generation is to maximize the total discounted utility. In this
process, the $t$--generation is not so much concerned with the
consumption of the future generations as it is with respect its
own consumption.

A common specification of $ d_t(s)$ is
\begin{equation}\label{Usualdisc}
 d_t(s)=\exp{\left(-\int_t^s\rho (r)\, dr\right)},
\end{equation}
where $\rho\in L^{\infty}([0,T])$. In this case $ d_t(s)$ is
Lipschitz in $(t,s)$, which is the present value at time $t$ of
one unit of utility at time $s\ge t$. The rate of discount is
$\rho$, and most often it is considered constant. Other popular
discount factors are those that depend only on the elapsed time, $
d_t(s)=\theta (s-t)$ for $s\ge t$, through a scalar function
$\theta$, with $\theta(0)=1$. As will be seen in Section
\ref{s:DPE}, the shape of the discount factor has a major effect
in the structure of the dynamic programming equation.

Let us define $\iota =( \nabla \ell)^{-1}$, the inverse of $\nabla
\ell$. Notice that by A1, both $\nabla \ell$ and $\iota$ are
continuous, and suprajective. We also consider
$\ell^*(p)=\sup_{u\in \R^n}\{p\cdot u-\ell(u)\}$, the Legendre
transform of $\ell$. Finally, let the $t$--Hamiltonian
\[
H_t(s,u,p) =  p\cdot u -  d_t(s)\ell (u).
\]
Throughout the paper, $\nabla $ denotes the gradient of a real
function, and $\nabla^2$ the Hessian Matrix. For a vector
function, $\nabla$ denotes the Jacobian matrix.

\section{Hopf--Lax formula}\label{s:Hopf}
A Hopf--Lax formula describes an infinite dimensional variational
problem as a finite dimensional one. In the present case, the
formula is a bit more involved than in the non--discounted case,
due to the non--autonomous term $ d_t(s)$. Notice also that the
problem at hand is different from the one with a non--autonomous
$\ell(s,\dot y(s))$, because the current date $t$ enters into the
definition.

Given $t\in [0,T]$, $t\le s\le T$, $x,\alpha\in \R^n$, consider
\begin{align*}
U_{t,\alpha}(s)&=\iota( d_t^{-1}(s)\nabla \ell(\alpha)),
\qquad ( d_t^{-1}=1/ d_t)\\
Y_{t,x,\alpha}(s)&=x+\int_t^s
\iota( d_t^{-1}(r)\nabla \ell(\alpha))\, dr,\\
V(x,t,\alpha)&=\int_t^T  d_t(s) \ell (U_{t,\alpha}(s))\, ds +
d_t(T) g\big(Y_{t,x,\alpha}(s)\big).
\end{align*}
Notice that $Y_{t,x,\alpha}(s)$ is absolutely continuous and
$Y_{x,t,\alpha}(t)=x$, thus it is an admissible arc, i.e. it
belongs to ${\AC}_{x,t}$. Observe also that $\dot
Y_{x,t,\alpha}(s) = U_{t,\alpha}(s)$.

We establish the following lemma to facilitate posterior
quotation. It is a consequence of assumption A1.
\begin{lemma}\label{Homeo}
For $x\in \R^n$, $t\in[0,T)$, $s\ge t$, the mappings $\alpha
\mapsto U_{t,\alpha}(s)$, $\alpha \mapsto Y_{t,x,\alpha}(s)$ are
of class $C^1$ and suprajective.
\end{lemma}

\begin{theorem} {\rm(Hopf--Lax formula with discount).}
If $x\in \R^n$ and $0\le t < T$, then the value function
$v=v(x,t)$ of the minimization problem \eqref{Problem} is given by
\begin{equation}\label{Hopf}
v(x,t)=\min_{p\in \R^n} \left\{\int_t^T  d_t(s) \ell (\iota(
d_t^{-1}(s)p))\, ds +  d_t(T) g\bigg(x+\int_t^T \iota(
d_t^{-1}(s)p)\, ds\bigg)\right\}.
\end{equation}

\end{theorem}

\begin{proof}
1. For any $\alpha \in \R^n$
\[
v(x,t)\le \int_t^T  d_t(s) \ell (\dot Y_{t,x,\alpha}(s))\, ds +
d_t(T) g(Y_{t,x,\alpha}(T)) = V(x,t,\alpha),
\]
and so
\[
v(x,t)\le \inf_{\alpha\in \R^n} V(x,t,\alpha).
\]

2. On the other hand, for an arbitrary function $y(s)$, $t\le s\le
T$, with $y(t)=x$, let $\overline \alpha$ be such that
\[
Y_{t,x,\overline \alpha}(T) = y(T).
\]
This is possible by Lemma \ref{Homeo}. For each $t$, $s$, $p$, the
Hamiltonian $H_t(\cdot, u,\cdot)$ is concave, thus for any
$\alpha$
\begin{equation}\label{conjugation}
H_t(s,U_{t,\alpha}(s),\nabla \ell(\alpha)) \ge H_t(s,\dot
y(s),\nabla \ell(\alpha)),
\end{equation}
since
\[
\left.\frac {\partial H_t(s,u,\nabla \ell(\alpha))}{\partial
u}\right|_{ u=U_{t,\alpha}(s)} = 0.
\]

Let $\alpha = \overline \alpha$ defined above. Integrating
\eqref{conjugation} between $t$ and $T$ and rearranging terms we
get
\begin{align*}
\int_t^T  d_t(s) \ell(U_{t,\overline \alpha}(s)))\, ds &\le \int
_t^T  d_t(s) \ell(\dot y(s))\, ds + \nabla \ell(\overline \alpha)
\int _t^T
(U_{t,\overline \alpha}(s))-\dot y(s))\, ds\\
&=\int _t^T  d_t(s) \ell(\dot y(s))\, ds + \nabla \ell(\overline
\alpha) \int
_t^T (\dot Y_{x,t,\overline \alpha}(s))-\dot y(s))\, ds\\
&= \int _t^T  d_t(s) \ell(\dot y(s))\, ds,
\end{align*}
because $Y_{x,t,\overline \alpha}(t)=x=y(t)$ and $Y_{x,t,\overline
\alpha}(T)=y(T)$. Adding $ d_t(T) g(Y_{x,t,\overline \alpha}(T)) =
d_t(T) g(y(T))$ to both terms of the above inequality we get that,
for any arc $y(s)$, there exist some $\overline \alpha$ such that
\[
V(x,t,\overline \alpha) \le \int_t^T  d_t(s) \ell (\dot y(s))\, ds
+  d_t(T) g(y(T)).
\]
Thus $\inf_ {\alpha\in \R^n} V(x,t,\alpha) \le v(x,t)$. Hence,
$\inf_ {\alpha\in \R^n} V(x,t,\alpha) = v(x,t)$. Finally, observe
that minimization with respect to $\alpha$ is equivalent of
minimization with respect to $p=\nabla \ell(\alpha)$.

3. The infimum is in fact attained, since the function
$V(\cdot,\cdot,\alpha)$ is continuous and inf--compact. Indeed,
$\lim_{|\alpha|\to \infty} |\alpha|^{-1} V(x,t,\alpha) = \infty$
due to the assumptions A1--A3 and Lemma \ref{Homeo}.
\end{proof}

Define $A(x,t) = \argmin_{\alpha\in \R^n} V(x,t,\alpha)$. Since
$\lim _{|\alpha|\to \infty} |\alpha|^{-1} V(x,t,\alpha) = \infty$,
$A$ is compact valued and upper semicontinuous correspondence.

The following corollary is along the lines of the above proof.

\begin{corollary}\label{Optimal}
If $x\in \R^n$ and $0\le t< T$, then for any selection
$\alpha(x,t)\in A(x,t)$, the arc $Y_{x,t,\alpha(x,t)}(s)$ is a
solution of problem \eqref{Problem}.
\end{corollary}


\begin{remark}
When $ d_t(s)=1$ for all $0\le t\le s\le T$, \eqref{Hopf} reduces
to the classical Hopf--Lax formula
\[
v(x,t)=\min_{\alpha \in \R^n} \left\{
(T-t)\ell\left(\frac{\alpha-x}{T-t}\right)+g(\alpha)\right\}.
\]
\end{remark}
%

For a locally Lipschitz function $f$, $\lip(f)$ will denote the
Lipschitz parameter of $f$ in a given compact set $K$, and $\bound
(f)$ will denote a bound of $|f|$ in that set. Notice that under
our assumptions $f=\iota,  d,  d^{-1}$ are locally Lipschitz.
\begin{theorem}\label{Lipschitz} {\rm (Lipschitz continuity).} The value function $v$ is
locally Lipschitz continuous in $\R^n\times [0,T]$ and
\[
v=g\quad \mbox{on } \R^n\times \{t=T\}.
\]
\end{theorem}
\begin{proof}
Let $x,\hat x\in K\subseteq \R^n$ with $K$ compact, and $t,\hat
t\in [0,T)$ and $\alpha\in \R^n$.

1. Let $s\in [0,T)$. Let us proceed to establish three Lipschitz
estimates.
\begin{align*}
|U_{\hat t,\alpha}(s)-U_{t,\alpha}(s)|&\le \lip(\iota)\nabla
\ell(\alpha) | d_{\hat t}^{-1} (s)- d_{
t}^{-1} (s)|\\
&\le \lip(\iota)\nabla \ell(\alpha)\lip( d^{-1})|\hat t-t| =
C|\hat t-t|.
\end{align*}
\begin{align*}
|Y_{\hat t,\hat x,\alpha}(s)-Y_{t,x,\alpha}(s)|&\le |\hat x- x| +
\left|\int_{\hat t}^T U_{\hat t,\alpha}(s)\, ds
- \int_t^T U_{t,\alpha}(s)|\, ds\right|\\
&\le  |\hat x- x| + \int_{\hat t \vee t}^T |U_{\hat
t,\alpha}(s)-U_{t,\alpha}(s)| \, ds + \int_{\hat t \wedge t}^{\hat
t
\vee t} |U_{\hat t \wedge t,\alpha}(s)|\, ds\\
&\le |\hat x- x| + TC |\hat t- t| + \bound(U) |\hat t- t|\\
&=|\hat x- x| + C |\hat t- t|.
\end{align*}
\begin{equation}\label{lipschitz}
\begin{aligned}
| d_{\hat t} (T) g(Y_{\hat t,\hat x,\alpha}(T))- d_{ t} (T)
g(Y_{t, x,\alpha}(T))| &\le | d_{\hat t} (T) - d_{
t} (T)||g(Y_{\hat t,\hat x,\alpha}(T)|\\
&\quad + d_{ t} (T)|g(Y_{\hat
t,\hat x,\alpha}(T))-g(Y_{t, x,\alpha}(T))|\\
& \le \lip( d)\bound(g) |\hat
t-t|\\
&\quad +\lip(g)\bound( d)\left(|\hat x- x|+C |\hat t-t|\right)\\
&=C(|\hat x- x|+ |\hat t-t|).
\end{aligned}
\end{equation}
In the above, we have used the same $C$ to denote several
constants.

3. Choose $\alpha\in A(x,t)$. Then, by definition of $v$ and
estimate \eqref{lipschitz}
\begin{align*}
v(x,t)-v(\hat x,\hat t) &\le V(\hat x,\hat t,\alpha)-V(x,t,\alpha)\\
&= d_{\hat t} (T) g(Y_{\hat t,\hat x,\alpha})- d_{ t} (T)
g(Y_{t,x,\alpha})\\
&\le C(|\hat x- x|+ |\hat t-t|).
\end{align*}
Reversing the role of $(\hat x, \hat t)$ and $(x,t)$ we get the
desired Lipschitz property.

4. Now, let $x\in \R^n$, $t<T$ and define $\delta_t=\int_t^T  d
_t(s) \, ds$. Choose $\alpha\in \R^n$ such that $\int_t^T
U_{t,\alpha}(s)\, ds = 0$; this is possible by virtue of Lemma
\ref{Homeo}. Let $b=\max_{s\in [t,T]} |U_{t,\alpha}(s)|$, and let
$\bound(\ell)$ be a bound of $|\ell|$ in $[-b,b]$. Then,
\begin{equation}\label{Ineq1}
v(x,t)\le \int_t^T  d _t(s) \ell (U_{t,\alpha}(s))\, ds +
 d_t(T)g(x)\le \bound (\ell) \delta_t +  d_t(T)g(x).
\end{equation}
Moreover,
\begin{align*}
v(x,t)&\ge  d_t(T) g(x) + \min_{\alpha \in \R^n} \left\{ -\lip(g)
\Big|\int_t^T U_{t,\alpha}(s)\, ds\Big| + \int_t^{T}
 d_t(s) \ell (U_{t,\alpha}(s))\, ds \right\}\\
&\ge  d_t(T) g(x) + \delta_t \min_{\alpha \in \R^n} \left\{
-\lip(g)\delta_t^{-1} \Big|\int_t^T U_{t,\alpha}(s)\, ds\Big| +
\ell\left(\delta_t ^{-1}\int_t^{T}  d_t(s) U_{t,\alpha} (s)\,
ds\right)\right\},
\end{align*}
by Jensen's inequality. Now, notice that for any $\alpha\in \R^n$
\[
\frac{\int_t^T U_{t,\alpha}(s)\, ds}{\int_t^T  d_t(s)
U_{t,\alpha}(s)\, ds}\rightarrow 1\qquad \mbox{as }t\to T^-
\qquad\mbox{(componentwise)}
\]
thus, for every $t$ close enough to $T$, there exists $\epsilon>0$
such that
\begin{align*}
v(x,t)&\ge  d_t(T) g(x) + \delta_t \min_{\alpha \in \R^n}
\left\{-\lip(g)(1+\epsilon)\delta_t^{-1}\Big|\int_t^T  d_t(s)
U_{t,\alpha}(s)\, ds\Big|\right.\\
&\hspace{5cm} + \left. \ell\left(\delta_t ^{-1}\int_t^{T}
 d_t(s) U_{t,\alpha} (s)\,
ds\right)\right\}\\
& =  d_t(T) g(x) - \delta_t \max_{z\in B}\max_{\alpha \in \R^n}
\left\{z\delta_t^{-1}\int_t^T  d_t(s) U_{t,\alpha}(s)\, ds\right.\\
&\hspace{5cm} - \left. \ell\left(\delta_t ^{-1}\int_t^{T}
 d_t(s) U_{t,\alpha} (s)\, ds\right)\right\},
\end{align*}
where $B=[-\lip(\ell)(1+\epsilon),\lip(\ell)(1+\epsilon)]$. Then
\begin{equation}\label{Ineq2}
v(x,t)\ge d_t(T) g(x) -
\delta_t\max_{z\in[-\lip(\ell)(1+\epsilon),\lip(\ell)(1+\epsilon)]}\ell^*(z),
\end{equation}
since $\alpha\longrightarrow \int_t^T  d_t(s) U_{t,\alpha}(s)\,
ds$ is suprajective. Thus, by \eqref{Ineq1} and \eqref{Ineq2}
\[
|v(x,t)- d_t(T)g(x)|\le C\delta_t
\]
for an appropriated constant $C$. Given that $ d_t(T)$ tends to 1
and $\delta_t$ tends to 0 as $t\to T$, we are done.
\end{proof}

\section{Dynamic programming equation}\label{s:DPE}
For any $y\in \AC_{x,t}$ and $t\le \tau \le T$, let $Y_{\tau}$
denote an optimal arc from initial condition $(y(\tau),\tau)$,
that is,
\[
Y_{\tau}(s)=Y_{y(\tau), \tau, \alpha(y(\tau),\tau)}(s),
\]
which exists by Corollary \ref{Optimal}.

Consider for $ t < \tau \le T $ the function
\[
W(x,t,\tau) = \int_{\tau}^T( d_t(s)- d_{\tau}(s))
\ell(\dot{Y}_{\tau}(s))\, ds + ( d_t(T)-
d_{\tau}(T))g(Y_{\tau}(T))
\]
\begin{lemma}\label{l:DPhalf} For every initial
condition $(x,t)$, admissible arc $y\in {\AC}_{x,t}$ and $t\le
\tau\le T$, we have
\begin{equation}\label{DPhalf}
v(x,t)\le \int_t^{\tau}  d_t(s) \ell (\dot y(s))\, ds +
v(y(\tau),\tau)
 + W(x,t,\tau).
\end{equation}
\end{lemma}
\begin{proof}
Let $y\in \AC_{x,t}$ be fixed but arbitrary. If $\tau=T$, then
$y(T)=Y_T(T)$ and $v(y(T),T)= d_T(T) g(y(T))=g(Y_T(T))$. Then
\eqref{DPhalf} reduces to $v(x,t)\le \int_t^{\tau}  d_t(s) \ell
(\dot y(s))\, ds +  d_t(T) g(y(T))$, which is true by the
definition of $v$. Now, suppose $\tau<T$. Let
$\alpha(y(\tau),\tau)\in A(y(\tau),\tau)$. Then, by Corollary
\ref{Optimal}
\[
\int_{\tau}^T  d_{\tau}(s) \ell (\dot {Y}_{\tau}(s)) +
 d_{\tau}(T) g(Y_{\tau}(T)) = v(y(\tau),\tau).
\]
Let us define the admissible arc $\tilde y\in {\AC}_{x,t}$ by
\[
\tilde y(s)=\left\{%
\begin{array}{ll}
    y(s), & \hbox{if $t\le s\le \tau$;} \\
    Y_{\tau}(s), & \hbox{if $\tau<s \le T$.} \\
\end{array}%
\right.
\]
We have
\begin{align*}
v(x,t)&\le \int_t^T  d_t(s)\ell(\dot {\tilde y}(s))\, ds +
 d_t(T) g(\tilde y(T))\\
&=\int_t^{\tau}  d_t(s)\ell(\dot y(s))\, ds + \int_{\tau}^T
 d_{\tau}(s) \ell(\dot{Y}_{\tau}(s))\, ds +  d_{\tau}
(T) g(Y_{\tau}(T))\\
&\quad  + \int_{\tau}^T ( d_t(s)-
d_{\tau}(s))\ell(\dot{Y}_{\tau}(s))\, ds +
 ( d_t(T)-  d_{\tau}(T)) g(Y_{\tau}(T))\\
& =\int_t^{\tau}  d_t(s)\ell(\dot y(s))\, ds  +
v(y(\tau),\tau)\\
&\quad + \int_{\tau}^T( d_t(s)- d_{\tau}(s))
\ell(\dot{Y}_{\tau}(s))\, ds + ( d_t(T)-  d_{\tau}(T))
g(Y_{\tau}(T)).
\end{align*}
\end{proof}

\begin{corollary}\label{C:DP} {\rm (Dynamic Programming).} For every initial
$(x,t)$ and $t\le \tau\le T$
\begin{equation}
\begin{aligned}
v(x,t)&=\min_{y\in \AC_{x,t}}\left\{\int_t^{\tau}  d_t(s) \ell
(\dot y(s))\, ds +  v(y(\tau),\tau)  \right\} + W(x,t,\tau).
\end{aligned}
\end{equation}
\end{corollary}
\begin{proof}
In fact \eqref{DPhalf} is an equality since an optimal arc is
attained for every initial condition $x,t$, by Corollary
\ref{Optimal}.
\end{proof}

Now consider the function
\begin{equation}\label{w}
w(x,t,\alpha)=-\int _t^T \frac{\partial d_t}{\partial t}(s) \,\ell
(U_{t,\alpha}(s))\, ds - \frac{\partial d_t}{\partial t}(T)\,
g(Y_{x,t,\alpha}(T)).
\end{equation}
The (generalized) dynamic programming equation is as follows. It
could be obtained for a more general optimal control problem with
some additional assumptions.
\begin{theorem}\label{th:DynProg} {\rm (Dynamic Programming Equation).}
Suppose that for every $t\le s\le T$, $ d_t(s)$,
$(\partial/\partial t)  d_t(s)$ are continuous in $t$ and summable
in $s$. Let $(x,t)$ be a point at which the value function $v$ is
differentiable. Then:
\begin{equation}\label{Th:GenHJ}
-v_t(x,t)+\ell ^*(-v_x(x,t))+w(x,t,\alpha(x,t)) = 0, \mbox{ in }
\R^n\times (0,T).
\end{equation}
\end{theorem}
\begin{proof}
1. By Lemma \ref{l:DPhalf}, if $t+h<T$
\begin{equation}\label{Th:DPhalf}
\begin{aligned}
v(x,t)-v(y(t+h),t+h))&\le \int_t^{t+h}  d_t(s) \ell (\dot
y(s))\, ds\\
&\quad + \int_{t+h}^T( d_t(s)- d_{t+h}(s))
\ell(\dot{Y}_{t+h}(s))\, ds\\
&\quad  + ( d_t(T)-  d_{t+h}(T)) g(Y_{t+h}(T)),
\end{aligned}
\end{equation}
for any $y\in {\AC}_{x,t}$.

2. The correspondence $A$ is compact valued and upper
semicontinuous, hence we can assume $\lim_{h\to 0^+} \alpha
(y(t+h),t+h)\in A(x,t)$; we denote the limit by $\alpha (x,t)$. By
continuity $\lim_{h\to 0^+} Y_{t+h}(s)=Y_t(s)$, and $\lim_{h\to
0^+}\dot Y_{t+h}(s) = \dot Y_{t}(s)$. Then,
\[
\lim_{h\to 0^+} h^{-1} ( d_t(T)-  d_{t+h}(T))
g(Y_{t+h}(T))=-\frac{\partial  d_t}{\partial t}(T) g(Y_{t}(T)),
\]
and
\[
\lim_{h\to 0^+} h^{-1} \int_{t+h}^T( d_t(s)- d_{t+h}(s)) \ell(\dot
{Y}_{t+h}(s))\, ds = -\int_t^T  \frac{\partial
 d_t}{\partial t}(s) \,\ell (\dot {Y}_{t}(s))\, ds.
\]

3. Taking limits in \eqref{Th:DPhalf}
\begin{align*}
\lim_{h\to 0^+} h^{-1}\big(v(x,t)-v(y(t+h),t+h)\big) &\le
\lim_{h\to 0^+} h^{-1}\int_t^{t+h}  d_t(s) \ell (\dot y(s))\,
ds\\
&  - \lim_{h\to 0} h^{-1} W(x,t,t+h)
\end{align*}
for every $y\in {\AC}_{x,t}$. This yields
\[
-v_t(x,t)-v_x(x,t)\cdot u - \ell(u) + w(x,t,\alpha (x,t)) \le 0
\]
for every $u\in \R^n$. Recalling the definition of $\ell^*$, this
is equivalent to
\[
-v_t(x,t)+\ell^*(-v_x(x,t)) + w(x,t,\alpha(x,t)) \le 0.
\]
4. To prove the equality, we use the same argument. Notice that
equality holds in \eqref{Th:DPhalf} for $y(s)=Y_t(s)$.
\end{proof}

\begin{remark}
If $\frac{\partial  d_t}{\partial t}(s) = \rho (t)
 d_t(s)$ for some continuous function $\rho$,
then equation \eqref{w} gives $w(x,t,\alpha(x,t))=\rho (t) v(x,t)$
hence, \eqref{Th:GenHJ} takes the form of a Hamilton--Jacobi
equation with a dissipation term
\[
-v_t(x,t)+\ell ^*(-v_x(x,t))+\rho(t) v(x,t) = 0, \mbox{ in
}\R^n\times (0,T).
\]
This happens if and only if \eqref{Usualdisc} holds, since we are
assuming $d_t(t)=1$ for each $t$.

\end{remark}

In the general case, the dynamic programming equation
\eqref{Th:GenHJ} has a complicated structure. Indeed, the optimal
arc itself enters the formulation as a non--local term, thus the
applicability of the equation should be taken with caution. In
contrast, the solution given in \eqref{Hopf} is simpler. This
stresses the usefulness of having a Hopf--Lax formula at hand.
Nevertheless, we can give a more amenable form to the dynamic
programming equation, close to classical standards, when assuming
that both the value function and function $g$ are differentiable.
This is the content of the next theorem.
\begin{theorem}\label{DPuseful}
With the same assumptions as in Theorem \ref{th:DynProg}, assume
further that $\nabla^2 \ell(u)$ is definite positive for every
$u\in \R^n$ and that $g$ is differentiable; then, the dynamic
programming equation \eqref{Th:GenHJ} is
\begin{equation}\label{DPgood}
-v_t(x,t)+\ell ^*(-v_x(x,t)) + w(x,t,\iota(-v_x(x,t))) = 0, \mbox{
in }\R^n\times (0,T).
\end{equation}

\end{theorem}
\begin{proof}
Since we are supposing $v$ is differentiable, the envelope theorem
applied to \eqref{Hopf} gives
\[
v_x(x,t)= d_t(T)\nabla g\left( Y_{t,x,\alpha}\right).
\]
On the other hand, $\alpha$ is an irrestricted minimum of $V$,
hence
\begin{align*}
0&= V_{\alpha}(x,t,\alpha)\\
& = \bigg(\nabla \ell (\alpha) + d_t(T)\nabla g
\big(Y_{t,x,\alpha}\big)\bigg)\, \bigg(\int_t^T
 d_t^{-1}(s)\nabla \iota\big( d_t^{-1}(s)\nabla \ell(\alpha)\big)\,
ds\bigg)\, \nabla^2\ell(\alpha)\\
&= \bigg(\nabla \ell (\alpha) + v_x(x,t)
\bigg)\, \bigg(\int_t^T
 d_t^{-1}(s)\nabla \iota\big( d_t^{-1}(s)\nabla \ell(\alpha)\big)\,
ds\bigg)\, \nabla^2\ell(\alpha).
\end{align*}
Since $\nabla^2\ell$ has maximal rank and $\nabla
\iota(\cdot)=(\nabla^2 \ell(\cdot))^{-1}$, the gradient of $V$
with respect to $\alpha$ is the null vector only if $\nabla
\ell(\alpha)=-v_x(x,t)$ and then $\alpha (x,t)=\iota (-v_x(x,t))$
at points of differentiability of $v$ (incidentally, this shows
that $\alpha$ must be unique at points of differentiability of
$v$). Plugging this value for $\alpha$ into $w(x,t,\alpha)$ we
reach the expression for the dynamic programming equation asserted
in the theorem.
\end{proof}

\section{Regularity of the value function}\label{s:Reg}
By Rademacher's Theorem, a locally Lipschitz function is almost
everywhere differentiable. Thus, by Theorem \ref{Lipschitz}, the
value function $v$, which is characterized by \eqref{Hopf} also
satisfies the dynamic programming equation almost everywhere.
Summarizing:
\begin{theorem} With the same assumptions as in Theorem \ref{DPuseful},
the function $v$ defined by the Hopf--Lax formula \eqref{Hopf} is
the value function \eqref{Problem}, which is locally Lipschitz
continuous in $\R^n\times [0,T)$, and solves the terminal value
problem (in a generalized sense)
\begin{equation}\label{TVP}
\left\{%
\begin{array}{rl}
    -v_t+\ell^*(-v_x)+w(x,t,\iota(-v_x))=0, &
    \hbox{a.e. $\in \R^n\times (0,T)$;}\\
    v=g, &\hbox{on $\R^n\times \{t=T\}$.} \\
\end{array}%
\right.
\end{equation}
\end{theorem}

In the conditions of the above theorem, for the particular case of
the Hamilton--Jacobi equation with dissipation we have
\begin{corollary}
Function $v$ given by \eqref{Hopf} with
$d_t(s)=\exp{\left(-\int_t^s \rho (r)\, dr\right)}$, $t\le s\le T$
is locally Lipschitz continuous in $\R^n\times [0,T)$, and solves
the terminal value problem \eqref{Diss}.
\end{corollary}

Now we establish some results on the smoothness of the value
function.
\begin{theorem}
With the same assumptions as in Theorem \ref{DPuseful}, suppose
further that $g$ is convex.
\begin{enumerate}
\item If $g$ is of class $C^1$, then the value function $v$ is
differentiable in $\R^n\times (0,T)$ and the minimizer $\alpha$ is
continuous.

\item If $g$ is of class $C^2$, then the value function $v$ is
also of class $C^2$ in $\R^n\times (0,T)$ and the minimizer
$\alpha$ is of class $C^1$ in $\R^n\times (0,T)$.
\end{enumerate}
\end{theorem}
\begin{proof}
The minimizers $\alpha$ in \eqref{Hopf} satisfy
\begin{equation}\label{Foc}
\nabla \ell(\alpha) +
 d_t(T)\nabla g\left(Y_{t,x,\alpha}(T)\right)=0,
\end{equation}
as shown in the proof of Theorem \ref{DPuseful}.

1. For $x, t$ fixed but arbitrary, the mapping $\alpha \mapsto
\nabla \ell(\alpha) +  d_t(T)\nabla
g\left(Y_{t,x,\alpha}(T)\right)$ is strictly monotone due to the
convexity of $g$ and the strict convexity of $\ell$, thus
$\alpha(x,t)$ is unique; as a correspondence, $\{\alpha(x,t)\}$ is
upper semicontinuous thus, as a function it is continuous. The
uniqueness of $\alpha$ leads to the differentiability of the value
function, by Danskin's Theorem.

2. The derivative of the L.H.S. of \eqref{Foc} with respect to
$\alpha$ is
\[
\bigg(I+ d_t(T)\nabla^2 g(Y_{x,t,\alpha}(T))\int_t^T
 d_t^{-1}(s)\nabla \iota\big( d_t^{-1}(s)\nabla \ell(\alpha)\big)\,
ds\bigg)\nabla^2\ell(\alpha),
\]
with $I$ being the identity matrix. Given our assumptions, this
vector has norm $\ge 1$ hence, \eqref{Foc} locally defines $\alpha
(x,t)$ of class $C^1$. This function is defined globally since the
mapping $\alpha \longmapsto \nabla \ell(\alpha) +
 d_t(T)\nabla g\left(Y_{t,x,\alpha}(T)\right)$ is proper because
\[
\lim_{\alpha\to\pm \infty} \Big(\nabla \ell(\alpha) +
 d_t(T)\nabla g\left(Y_{t,x,\alpha}(T)\right)\Big) = \pm \infty.
\]
By the envelope theorem, $v_x(x,t)= \nabla
g(Y_{x,t,\alpha(x,t)}(T))$ is of class $C^1$ hence, by the dynamic
programming equation \eqref{DPgood}, $v_t$ is also of class $C^1$.
\end{proof}

Finally, a result concerning the monotonic behavior of $\alpha$ in
the scalar case.
\begin{theorem}\label{C2}
With the same assumptions as in Theorem \ref{th:DynProg} with
$n=1$, assume further that $g$ is convex and of class $C^2$.
\begin{enumerate}
\item For each time $0<t<T$, there exists for all but at most
countably many values of $x\in\R$ a unique point $\alpha(x,t)$
where the minimum in \eqref{Hopf} is attained.

\item The mapping $x\mapsto \alpha (x,t)$ is nondecreasing.
\end{enumerate}
\end{theorem}
\begin{proof}
Since $v$ is locally Lipschitz, it is differentiable almost
everywhere, thus the minimum $\alpha$ is unique almost everywhere.
On the other hand, the crossed derivative of $V(x,t,\alpha)$ with
respect to $x$ and $\alpha $ is
\[
 d_t(T)\ell''(\alpha)g''(Y_{t,x,\alpha})\int_t^T
 d_t^{-1}(s)\iota'( d_t^{-1}(s)\nabla \ell(\alpha))\, ds
\ge 0.
\]
Hence, $(x,\alpha)\mapsto V(x,t,\alpha)$ is supermodular, and by
Topkis' Theorem, \cite{Topkis}, $A(x,t)$ is a nonempty compact
sublattice which admits a lowest element, which we denote again by
$\alpha(x,t)$, satisfying $\alpha(x_2,t) \ge \alpha(x_1,t)$
whenever $x_2>x_1$. Then the mapping $x\mapsto \alpha(x,t)$ is
non--decreasing and thus continuous for all but at most countably
many $x$.
\end{proof}

\vfill

\noindent Juan Pablo Rinc\'on--Zapatero\\
Departamento de Econom\'{\i}a\\
Universidad Carlos III de Madrid,\\
E-28903 Getafe, Spain\\
\textit{E-mail:}\texttt{\,jrincon@eco.uc3m.es}


\begin{thebibliography}{8}
\bibitem{Barro} Barro, R., Ramsey meets Laibson in the
neoclassical growth model, \emph{Quat. J. Econ.} \textbf{114}
(199), 1125--1152.
\bibitem{BarronJensenLiu} Barron, E.N., R. Jensen, W. Liu,
Hopf--Lax--type formula fo $u_t+H(u,Du)=0$, \emph{J. Diff. Equat.}
\textbf{126} (1996), 48--61.
\bibitem{Evans} Evans, L.C., \emph{Partial differential
equations}, American Mathematical Society, Rhode Island, 1998.
\bibitem{Karp} Karp L., Non--constant discounting in continuous
time, \emph{J. Econ. Theory} \textbf{132} (2007), 557--568.
\bibitem{Hopf} Hopf, E., Generalized solutions of nonlinear equations
of first order, \emph{J. Meth. Mech.} \textbf{14} (1965),
951--973.
\bibitem{Lax} Lax, P.D., Hyperbolic systems of conservation laws
II, \emph{Commun. Pure Appl. Math.} \textbf{10} (1957), 537--566.
\bibitem{MarinNavas} Mar{\'\i}n--Solano, J., J. Navas, Non--constant
discounting in finite horizon: The free terminal time case,
\emph{J. Econ. Dyn. Control} \textbf{33} (2009), 666-675.
\bibitem{Topkis} Topkis, D., Minimizing a submodular function on a
lattice, \emph{Operations Research} \textbf{26}, 305--321.
\bibitem{VanTsujiSon}
Van, T.D., M. Tsuji, N.D. Thai Son, \emph{The characteristic
method and its generalizations for first--order nonlinear partial
differential equations}, Chapman \& Hall, New York, London.
\end{thebibliography}
\end{document}